\documentclass[11pt,leqno]{amsart}
\usepackage{enumerate}
\usepackage{amssymb,amsmath,amsthm,amsfonts}
\usepackage {latexsym}
\usepackage{bbm}

\allowdisplaybreaks
\newcommand{\nc}{\newcommand}
\nc{\les}{\lesssim}
\nc{\nit}{\noindent}
\nc{\nn}{\nonumber}
\nc{\D}{\partial}
\nc{\diff}[2]{\frac{d #1}{d #2}}
\nc{\diffn}[3]{\frac{d^{#3} #1}{d {#2}^{#3}}}
\nc{\pdiff}[2]{\frac{\partial #1}{\partial #2}}
\nc{\pdiffn}[3]{\frac{\partial^{#3} #1}{\partial{#2}^{#3}}}
\nc{\abs}[1] {\lvert #1 \rvert}
\nc{\cAc}{{\cal A}_c}
\nc{\cE}{{\cal E}}
\nc{\cF}{{\cal F}}
\nc{\cP}{{\cal P}}
\nc{\cV}{{\cal V}}
\nc{\cQ}{{\cal Q}}
\nc{\cGin}{{\cal G}_{\rm in}}
\nc{\cGout}{{\cal G}_{\rm out}}
\nc{\cO}{{\cal O}}
\nc{\Lav}{{\cal L}_{\rm av}}
\nc{\cL}{{\cal L}}
\nc{\cB}{{\cal B}}
\nc{\cZ}{{\cal Z}}
\nc{\cR}{{\cal R}}
\nc{\cT}{{\cal T}}
\nc{\cY}{{\cal Y}}
\nc{\cX}{{\cal X}}
\nc{\cXT}{{{\cal X}(T)}}
\nc{\cBT}{{{\cal B}(T)}}
\nc{\vD}{{\vec \mathcal{D}}}
\nc{\efield}{\mathcal{E}}
\nc{\vE}{{\vec \efield}}
\nc{\vB}{{\vec \mathcal{B}}}
\nc{\vH}{{\vec \mathcal{H}}}
\nc{\ty}{{\tilde y}}
\nc{\tu}{{\tilde u}}
\nc{\tV}{{\tilde V}}
\nc{\Pc}{{\bf P_c}}
\nc{\bx}{{\bf x}}
\nc{\bX}{{\bf X}}
\nc{\bXYZ}{{\bf XYZ}}
\nc{\bY}{{\bf Y}}
\nc{\bF}{{\bf F}}
\nc{\bS}{{\bf S}}
\nc{\dV}{{\delta V}}
\nc{\dE}{{\delta E}}
\nc{\TT}{{\Theta}}
\nc{\dPsi}{{\delta\Psi}}
\nc{\order}{{\cal O}}
\nc{\Rout}{R_{\rm out}}
\nc{\eplus}{e_+}
\nc{\eminus}{e_-}
\nc{\epm}{e_\pm}
\nc{\eps}{\varepsilon}
\nc{\vnabla}{{\vec\nabla}}
\nc{\G}{\Gamma}
\nc{\w}{\omega}
\nc{\mh}{h}
\nc{\mg}{g}
\nc{\vphi}{\varphi}
\nc{\tlambda}{\tilde\lambda}
\nc{\be}{\begin{equation}}
\nc{\ee}{\end{equation}}
\nc{\ba}{\begin{eqnarray}}
\nc{\ea}{\end{eqnarray}}

\nc{\g}{\gamma}
\nc{\ol}{\overline}

\newtheorem{theorem}{Theorem}[section]
\newtheorem{lemma}[theorem]{Lemma}
\newtheorem{prop}[theorem]{Proposition}

\newtheorem{rmk}[theorem]{Remark}

\nc{\pT}{\partial_T}
\nc{\pz}{\partial_z}
\nc{\pt}{\partial_t}
\nc{\la}{\langle}
\nc{\ra}{\rangle}
\nc{\infint}{\int_{-\infty}^{\infty}}
\nc{\halfwidth}{6.5cm}
\nc{\figwidth}{10cm}
\newcommand{\f}{\frac}

\nc{\nlayers}{L} \nc{\nsectors}{M}
\nc{\indicator}{\mathbf{1}}
\nc{\Rhole}{R_{\rm hole}}
\nc{\Rring}{R_{\rm ring}}
\nc{\neff}{n_{\rm eff}}
\nc{\Frem}{F_{\rm rem}}
\nc{\R}{\mathbb R}
\nc{\Z}{\mathbb Z}
\nc{\DD}{\Delta}
\nc{\cD}{\mathcal D}
\nc{\lnorm}{\left\|}
\nc{\rnorm}{\right\|}
\nc{\rnormp}{\right\|_{\ell^{p,\eps}}}
\nc{\rar}{\rightarrow}
\date{\today}
\begin{document}

\begin{abstract}

Let $H=-\Delta+V$ be a Schr\"odinger operator on $L^2(\R^4)$ with real-valued potential $V$, and let $H_0=-\Delta$.  If $V$ has sufficient pointwise decay, the wave operators
$W_{\pm}=s-\lim_{t\to \pm\infty} e^{itH}e^{-itH_0}$ are
known to be bounded on $L^p(\R^4)$ for all $1\leq p\leq \infty$ if zero is not an eigenvalue or resonance,
and on $\f43<p<4$ if zero is an eigenvalue but not a resonance.
We show that in the latter case, the wave operators are also bounded on $L^p(\R^4)$ for $1\leq p\leq \f43$ by direct examination of the integral kernel
of the leading terms.
Furthermore,
if 
$\int_{\R^4} xV(x) \psi(x) \, dx=0$ for all zero energy eigenfunctions $\psi$, then the wave operators
are bounded on $L^p$ for $1 \leq p<\infty$.

\end{abstract}

\title[$L^p$ boundedness of wave operators]{\textit{On the $L^p$ boundedness of wave operators for four-dimensional Schr\"odinger Operators with a threshold eigenvalue}}

\author[M.~J. Goldberg, W.~R. Green]{Michael Goldberg and William~R. Green}

\address{Department of Mathematics \\
University of Cincinnati\\
Cincinnati, OH 45221-0025}
\email{Michael.Goldberg@uc.edu}
\address{Department of Mathematics\\
Rose-Hulman Institute of Technology \\
Terre Haute, IN 47803 U.S.A.}
\email{green@rose-hulman.edu}

\thanks{This  work  
was  partially  supported  by  a  grant  from  the  Simons  Foundation (Grant  Number 281057
to  the first author.)}

\maketitle

\section{Introduction}

This work is inspired by a conjecture of Jensen and Yajima in
\cite{JY4} on the range of $p$ for which the Schr\"odinger wave operators are $L^p(\R^4)$ bounded in the presence of a
threshold eigenvalue.  In \cite{JY4} it was proven that the
wave operators are bounded on $L^p(\R^4)$ for $\f43<p<4$ in the presence of a threshold eigenvalue and conjectured that
the boundedness is true on $1\leq p\leq \f43$ as well. Recent works of Yajima, \cite{YajNew} and the authors \cite{GGwaveop} prove a similar results extending the lower range of $p$ for which $L^p(\R^n)$ bounds hold in dimensions $n>4$.  In this article we prove the conjectured bounds in four dimensions, and also show that the range of $p$ can be extended upwards under certain orthogonality conditions on the zero energy eigenspace.

Let $H=-\Delta+V$ be a Schr\"odinger operator
with a real-valued potential $V$ and
$H_0=-\Delta$.  If $V$  satisfies
$|V(x)|\les \la x\ra^{-2-}$, 
the spectrum of $H$ is the absolutely continuous
spectrum on $[0,\infty)$ and a finite collection of
non-positive eigenvalues,\ \cite{RS1}.
The wave operators are defined by the strong limits on $L^2(\R^n)$
\begin{align}
W_{\pm}f =\lim_{t\to\pm \infty} e^{itH}e^{-itH_0}f.
\end{align}
Such limits are known to exist and be asymptotically
complete  for a wide class of
potentials $V$.  
Furthermore, one has
the identities 
\begin{align}
W_\pm^* W_\pm=I, \qquad W_\pm W_\pm^*=P_{ac}(H),
\end{align}
with $P_{ac}(H)$ the projection onto the absolutely continuous
spectral subspace associated with the Schr\"odinger operator $H$.

We say that zero energy is regular if there are no zero
energy eigenvalues or resonances.
There is a zero energy eigenvalue if there is a
solution to $H\psi =0$ with $\psi\in L^2(\R^n)$, and a
resonance if $\psi\notin L^2(\R^n)$ instead belongs to a nearby space whose precise definition depends on the spatial dimension  $n\leq 4$.  In dimension $n=4$, a resonance satisfies $\la \cdot \ra^{-\epsilon} \psi \in L^2(\R^4)$ for any $\epsilon>0$.

There is a long history of results on the existence and
boundedness of the wave operators.  Yajima
has established $L^p$ and $W^{k,p}$ boundedness of the
wave operators for the full range of $1\leq p\leq \infty$ in \cite{YajWkp1,YajWkp3} in all
dimensions $n\geq 3$, provided that zero energy is regular
under varying assumptions on the potential $V$.  The sharpest
result for $n=3$ was obtained in \cite{Bec} by Beceanu .

If zero is not regular,
in general, the range of $p$ on which the wave operators are bounded shrinks.
When
$n>4$ it was shown in \cite{Yaj,FY} that the wave operators are bounded on $L^p(\R^n)$ when
$\frac{n}{n-2}<p<\frac{n}{2}$.
Independent results of 
Yajima \cite{YajNew,YajNew2} and the authors \cite{GGwaveop} then brought the lower bound on $p$ down to $1<p<\frac{n}{2}$ (with the authors obtaining $p=1$ as well), and found conditions under which the upper bound may be raised.  When $n=3$ Yajima \cite{YajNew3} showed that the wave operators are bounded on $1\leq p<3$ in the case of a zero energy eigenvalue and on $1<p<3$ but not $p=1$ in the case of a zero energy resonance.  This extended the range $\f32<p<3$ proven in \cite{Yaj}.  Finally, with conditions as in \cite{GGwaveop,YajNew2}, the full range of $1\leq p\leq \infty$ is recovered when $n=3$.

Results are also known in one dimension.  In \cite{Wed}, Weder showed that the wave operators are bounded on $L^p(\R)$ for $1<p<\infty$, and that the endpoint $p=1$ is possible under certain conditions on the Jost solutions, but is weak-type in general.  Further work in one dimension was done by
D'Ancona and Fanelli in \cite{DF}.    To the best of the
authors' knowledge, there are no results in the literature
when zero is not regular and $n=2$.  

An important property of the wave operators is the
intertwining identity,
\begin{align}\label{eqn:intertwine}
f(H)P_{ac}(H)=W_{\pm}f(-\Delta)W_{\pm}^*,
\end{align}
which is valid for any Borel function $f$.  Using this, one can deduce properties of the operator $f(H)$ from the
simpler operator $f(-\Delta)$, provided one has
control on mapping properties of the wave operators
$W_\pm$ and $W_\pm^*$.
In four dimensions, boundedness of the 
wave operators on $L^p(\R^4)$ for a given $p\geq 2$ imply the dispersive estimates
\begin{equation} \label{eqn:dispersive}
\|e^{itH}P_{ac}(H)\|_{L^{p'}\to L^p}\les
|t|^{-2+\frac{4}{p}}.
\end{equation}
Here $p^\prime$ is the conjugate exponent satisfying
$\frac{1}{p}+\frac{1}{p^\prime}=1$.  

There has been much work on dispersive estimates
for the Schr\"odinger evolution with zero energy obstructions in recent years by Erdo\smash{\u{g}}an, Schlag, Toprak and the authors in various combinations, see \cite{ES2,goldE,EG,EGG,GGodd,GGeven,GT} in which $L^1(\R^n)\to L^\infty(\R^n)$ were studied for all $n>1$.  Estimates in $L^p(\R^n)$ are obtained by interpolating these results with the  $L^2$ conservation law.  These works
have roots in previous work of \cite{JSS}, and also in
\cite{Jen2,Mur} where the dispersive estimates were studied as operators on weighted
$L^2(\R^n)$ spaces.

Our main result confirms the conjecture of Jensen and Yajima,
as well as extending the range of $p$ upward under certain conditions on the zero energy eigenspace.  
\begin{theorem}\label{thm:main}
	
	Let
	$\sigma>\frac{2}{3}$.  Assume that
	$|V(x)|\les \la x\ra^{-\beta}$ for 
	$\beta>6$,
	\begin{align}\label{Vfour}
	\mathcal F\big(\la \cdot \ra^{2\sigma} V \big)
	\in L^{\f32}(\R^4),
	\end{align}
	and $H = -\Delta+V$ has an eigenvalue at zero with no resonance.
	\begin{enumerate}[i)]
		\item \label{reg result}
		The wave operators extend to bounded operators on $L^p(\R^4)$
		for all $1 \leq p < 4$.

		\item \label{Pvx result} If  
		$\int_{\R^4} xV(x) \psi(x) \, dx=0$
		for all zero energy
		eigenfunctions $\psi$, then the wave operators
		extend to bounded operators on $L^p(\R^4)$ for all
		$1 \leq p<\infty$.		
		
	\end{enumerate}
	
\end{theorem}

We expect that the endpoint case $p=\infty$
holds if one has the additional cancellation $\int_{\R^4} x^2 V(x) \psi(x) \, dx=0$ and slightly more
decay on the potential, see Remark~\ref{PVx2 rmk} below.  The corresponding dispersive estimate
from $L^1(\R^4)$ to $L^\infty(\R^4)$
was recently shown to be true even without this extra hypothesis by the second author and Toprak in~\cite{GT}.

We prove Theorem~\ref{thm:main} for $W=W_-$,
the proof for $W_+$ is identical up to complex conjugation.
Define $R_0^\pm(\lambda^2) := \lim\limits_{\eps \to 0^+} (H_0 - (\lambda \pm i\eps)^2)^{-1}$
and $R_V^+(\lambda^2) := \lim\limits_{\eps \to 0^+} (H - (\lambda +i\eps)^2)^{-1}$ 
as the free and perturbed resolvents, respectively.  These operators are well-defined on polynomially
weighted $L^2(\R^4)$ spaces due to the limiting absorption
principle of Agmon, \cite{agmon}.
We use the  stationary representation of the wave operator
\begin{align}\label{stat rep}
Wu &=u-\frac{1}{\pi i} \int_0^\infty \lambda R_V^+(\lambda^2)V
[R_0^+(\lambda^2)-R_0^-(\lambda^2)]\, u\, d\lambda \\
&= u-\frac{1}{\pi i} \int_0^\infty \lambda \big[R_0^+(\lambda^2) -R_0^+(\lambda^2)VR_V^+(\lambda^2)\big]V
[R_0^+(\lambda^2)-R_0^-(\lambda^2)]\, u\, d\lambda.\nn
\end{align}
The 
identity
$R_V^+(\lambda^2)=R_0^+(\lambda^2)-R_0^+(\lambda^2)VR_V^+(\lambda^2)$ relates the resolvents and justifies
the equality in~\eqref{stat rep}.
In dimension $n=4$, the free
resolvent operators $R_0^\pm(\lambda^2)$ are bounded as $\lambda \to 0$, as
are the perturbed resolvents $R_V^\pm(\lambda^2)$ if zero is regular.  When zero is not regular,
the perturbed resolvent
becomes singular as $\lambda\to 0$.  

We divide the representation for $W$ in \eqref{stat rep} into `high' and `low' energy parts, by writing
$W=W\Phi^2(H_0)+W \Psi^2(H_0)$ with $\Phi, \Psi \in C_0^\infty(\R)$ smooth cut-off functions that satisfy
$\Phi^2(\lambda)+\Psi^2(\lambda)=1$ with $\Phi(\lambda^2)=1$ for $|\lambda|\leq \lambda_0/2$ and
$\Phi(\lambda^2)=0$ for $|\lambda|\geq \lambda_0$ for a
suitable constant $0<\lambda_0\ll 1$.  This allows us to 
write $W=W_<+W_>$, with $W_<$ the `low energy' portion
of the wave operator and $W_>$ the `high energy' portion.
The high energy term $W_>$ is controlled in~\cite{JY4}, and this argument remains valid when zero energy is a resonance or eigenvalue, whose effects are limited to only an arbitrary small neighborhood of zero energy. 
Our technical analysis proceeds much in the same vein as \cite{GGwaveop}.  We isolate the leading
order terms of $W_<$ caused by the singularity of $R_V^\pm(\lambda^2)$ near $\lambda=0$ and determine their $L^p(\R^4)$ operator bounds through a careful, pointwise analysis of their integral kernels.


The next section introduces some ideas for controlling the size of the
leading order expression in $W_<$ when there is a zero energy eigenvalue,
and describes how certain pointwise bounds on the integral kernel correspond
to operator estimates in $L^p(\R^4)$.  In Section~\ref{sec:main} we
calculate the leading order expression in detail and complete the proof of
Theorem~\ref{thm:main}, modulo a number of integral estimates that are
stated in Appendix~\ref{sec:appA}.  The discussion concludes in Section~\ref{sec:res} with some remarks about
the case where a zero energy resonance is present.


\section{Preliminary Steps}

In four dimensions it is well-known (see \cite{JY4,EGG,GT}) that if there is a zero energy eigenvalue but no zero energy resonance, then  the perturbed resolvent $R_V^+(\lambda^2)$ in~\eqref{stat rep} has a pole of order two
whose residue is the finite-rank projection $P_e$ onto the eigenspace.  Furthermore, each zero energy eigenfunction $\psi$ has the
cancellation property
\be
\int_{\R^4} V(x) \psi(x)\, dx=0,\label{eqn:PV1}
\ee 
which we express in the shorthand $P_eV1=0$.  This fact is crucial in obtaining the full range of $p$ as it permits improved estimates for the most singular terms in the expansion of $W_<$ that dictate the allowable range of $p$.  The extra cancellation condition in part~\ref{Pvx result}) of Theorem~\ref{thm:main}, namely that $\int_{\R^4} x_j V\psi(x) \,dx = 0$ for each $j \in [1,n]$, will be called $P_eVx=0$.

Using the low energy expansion for $W_{<}$ in \cite{JY4}, the leading term is 
given by the operator
\begin{equation} \label{eqn:Ws}
W_{s}=\frac{1}{\pi i} \int_0^\infty R_0^+(\lambda^2)
VP_eV(R_0^+(\lambda^2)-R_0^-(\lambda^2))
\tilde \Phi(\lambda)\lambda^{-1}\, d\lambda.
\end{equation}
Here $\tilde \Phi(\lambda)\in C_c^\infty(\R)$ is such 
that $\tilde \Phi(\lambda)\Phi(\lambda^2)=\Phi(\lambda^2)$.
In the absence of a zero energy resonance, the next operator we need to control is
\begin{equation} \label{eqn:Wlog}
W_{log}=\frac{1}{\pi i} \int_0^\infty R_0^+(\lambda^2)
L_1 (R_0^+(\lambda^2)-R_0^-(\lambda^2))
\tilde \Phi(\lambda)\lambda \log(\lambda )\, d\lambda,
\end{equation}
where $L_1$ is a finite rank operator with kernel
$$
L_1(x,y)=\sum_{j,k=1}^d a_{jk} V\psi_j(x) V\psi_k(y),\qquad\qquad a_{jk}\in \R,
$$
and $\{\psi_j\}_{j=1}^d$ form an orthonormal basis for the zero energy eigenspace.

One can show that the remaining terms in the expansion of $W_<$ are
better behaved.
Thus the estimates on $W_{s}$, and to a lesser extent $W_{log}$,
dictate the mapping properties of $W_<$ itself.
The presence or absence of threshold eigenvalues has little effect
on properties of the resolvent outside a small neighborhood of
$\lambda = 0$, so the estimates for $W_>$ are unchanged.
Therefore our primary effort will be to control the mapping
properties of the operator $W_s$.

\begin{prop} \label{prop:PV1}
	Assume that $|V(z)|\les \la z\ra^{-\delta}$ for some $\delta>0$.
	\begin{enumerate}[i)]
		\item If $\delta>4$, then
		$W_s$ is bounded on $L^p(\R^4)$ for $ 1 \leq p<4$.\label{prop:PV1pt1}
		
		\item If $\delta>6$, and the zero energy eigenspace satisfies the 
		cancellation condition $P_eVx=0$, then $W_s$ is bounded on $L^p(\R^4)$
		for $1 \leq p < \infty$.\label{prop:PV1pt2}
		
	\end{enumerate}
	
\end{prop}


The proof of this proposition is based on pointwise bounds for  the integral kernel $K(x,y)$  of the operator $W_s$. 
To get started,
the kernel of $W_{s}$ is a sum of integrals of the form
\begin{align}\label{eqn:Ws22}
K^{jk}(x,y) = \int_0^\infty \iint_{\R^{8}} R_0^{+}(\lambda^2)(x,z)& V(z)\psi_{j}(z)
V(w) \psi_{k}(w) \\  
&(R_0^+-R_0^-)(\lambda^2)(w,y)
\frac{\tilde\Phi(\lambda)}{\lambda} \, dwdz\, d\lambda, \nn
\end{align}
where the functions $\{\psi_j\}_{j=1}^N$ form an orthonormal basis for the zero energy eigenspace.  

For the remainder of the paper, we omit the subscripts 
on the eigenfunctions as our calculations will be
satisfied for any such $\psi$.  Our estimates are 
stated for an operator kernel $K(x,y)$ with the understanding that
each $K^{jk}(x,y)$ obeys the same bounds.  
To analyze the kernel, we split into three regimes based on the relative size of $|x|$ and $|y|$: $|x| > 2|y|$, $|y| > 2|x|$, and $|x| \approx |y|$.
The operator estimates resulting from a typical pointwise bound in each regime
are summarized in the following three lemmas.

\begin{lemma} \label{lem:xbig}
	Suppose $K(x,y)$ is an integral kernel supported in the region of $\R^8$ where $|x| > 2|y|$, and
	$|K(x,y)| \les \la x\ra^{-4-\alpha}$.  
	\begin{enumerate}[i)]
		
		\item If $\alpha>0$, then $K$ defines a bounded operator on
		$L^p(\R^4)$ for all $1 \leq p \leq \infty$.  
		
		\item If $\alpha=0$, then $K$ defines a
		bounded operator on $L^p(\R^4)$ for all $1 < p \leq \infty$ but it may not be bounded on $L^1(\R^4)$.
		
	\end{enumerate}
	
\end{lemma}

\begin{lemma} \label{lem:ybig}
	Suppose $K(x,y)$ is an integral kernel supported in the region of $\R^8$ where $|y| > 2|x|$, 
	and $|K(x,y)| \les \la x\ra^{-\gamma} \la y\ra^{-\beta}$ for some $0 < \beta \leq 4$ and
	$\gamma \geq 4-\beta$.
	Then $K$ defines a bounded operator on $L^p(\R^4)$ for all $1 \leq p < \frac{4}{4-\beta}$.
\end{lemma}

\begin{lemma} \label{lem:xapproxy}
	Suppose $K(x,y)$ is an integral kernel supported in the region of $\R^8$ where $\frac12|x| \leq |y| \leq 2|x|$,
	and $|K(x,y)| \les \la x\ra^{-3-\alpha} \la|x| - |y|\ra^{-1-\beta}$ for some $\alpha, \beta \geq 0$.
	\begin{enumerate}[i)]
		
		\item If $\alpha+\beta > 0$, then $K$ defines a bounded operator on $L^p(\R^n)$ for all $1 \leq p \leq \infty$.
		
		\item If $\alpha = \beta = 0$, it is not guaranteed that the operator is bounded on $L^p(\R^4)$ for any $p$.
		
	\end{enumerate}

\end{lemma}

The proofs of Lemmas~\ref{lem:xbig}-\ref{lem:xapproxy} are technical and are provided in Appendix~\ref{sec:appA}.
There are two main factors which control the integrability and size of~\eqref{eqn:Ws22}.
The integral in $\lambda$ is highly oscillatory due to the presence of Bessel functions in the
formula for the free resolvents $R_0^\pm(\lambda^2)$.  In the $w$ and $z$ variables,
decay of the potential $V$ and the eigenfunctions $\psi$ effectively localize most integrals to a 
neighborhood of the origin.  A representative example of each kind of estimate are as follows.

\begin{lemma}[\cite{GGwaveop}, Lemma 2.2] \label{lem:lambdaInt}
	Let $R_0^\pm(\lambda^2,A)$ denote the convolution kernel of $R_0^\pm(\lambda^2)$
	evaluated at a point with $|x-y| = A$.  For each $j \geq 0$,
	\begin{equation}
	\int_0^\infty R_0^+(\lambda^2, A)\partial_B^j\big(R_0^+ - R_0^-\big)(\lambda^2,B)
	\lambda^{-1}\tilde\Phi(\lambda)\, d\lambda
	\les \frac{1}{A^{2}\la A+B\ra \la A-B\ra^{1+j}}.
	\end{equation}
\end{lemma}

Two variations of this bound, with different powers of $\lambda$ in the integrand and different placement of the partial
derivatives, will also be required.  These are stated as Lemmas~\ref{lem:logssuck} and~\ref{lem:leftcanc1} in the Appendix.  The
first one is proved in~\cite{GGwaveop} along with Lemma~\ref{lem:lambdaInt}.

\begin{lemma}[\cite{GGwaveop}, Lemma 4.3] \label{lem:A1}
	
	Let $\beta \geq 1$ and $0 \leq \alpha < n-1$.  If $N\geq n+\beta$, 
	then for each fixed constant $R\geq 0$, we have the bound
	$$
	\int_{\R^n}\frac{\la z\ra^{-N}}{|x-z|^\alpha\la |x-z| + R\ra \la |x-z|-R\ra^\beta} \, dz
	\les \frac{1}{\la x\ra^\alpha \la |x| + R\ra \la R-|x| \ra^\beta}.
	$$
	
\end{lemma}
The $\la z\ra^{-N}$ decay in the numerator of Lemma~\ref{lem:A1} is achieved by the
combined decay of the potential $V(z)$ and  eigenfunctions $\psi(z)$.
Eigenfunctions at zero energy have a characteristic rate of decay which comes
from the Green's function of the Laplacian and the additional cancellation property $P_eV1=0$.
\begin{lemma}\label{lem:efn decay}
	
	If $|V(x)|\les \la x\ra^{-3- }$, and
	$\psi$ is a zero energy eigenfunction, then
	$|\psi(x)|\les \la x\ra^{-3}$.
	
\end{lemma}

\begin{proof}
	By definition, every eigenfunction $\psi(x)$ satisfies $\Delta \psi = V\psi$.
	After applying the Green's function of the Laplacian to both sides,
	$\psi(x) = C ( |x|^{-2}* V\psi)$.
	Using $P_eV1=0$ allows us to write
	$$
	\psi(x) = C\int_{\R^4} V\psi(y)\Big(\frac{1}{|x-y|^2} - \frac{1}{|x|^2}\Big)\, dy.
	$$
	When $|y| < \frac12|x|$, we have the bound $\big||x-y|^{-2} - |x|^{-2} \big|
	\les |y|\,|x|^{-3}$, hence
	\begin{multline*}
	|\psi(x)| \les \int_{|y| < \frac12|x|} \frac{|y|\,|V\psi(y)|}{|x|^3}\,dy
	+ \int_{|x-y| < \frac12|x|} \frac{|V\psi(y)|}{|x-y|^2}\,dy\\
	+ \int_{|y| > \frac12|x|} \frac{|V\psi(y)|}{|x|^2}\,dy
	\les \frac{1}{|x|^3},
	\end{multline*}
	provided $|V\psi(y)| \les \la y\ra^{-5-}$.  Since we have the {\it a priori}
	estimate that $|\psi(x)| \les \la x\ra^{-2}$, see Lemma~2.1 in \cite{GGwaveop}, it suffices to assume that $|V(x)|
	\les \la x\ra ^{-3-}$.
\end{proof}


\section{Main Estimates for $K(x,y)$} \label{sec:main}
The technical tools required to prove Proposition~\ref{prop:PV1} are somewhat different depending on the relative size of $|x|$ and $|y|$ in the integral kernel $K(x,y)$.  In order to employ Lemmas~\ref{lem:xbig}-\ref{lem:xapproxy}, we proceed by making separate estimates where $|y|$ is greater than, approximately equal to, or smaller than $|x|$.  The region where $|y| > 2|x|$ plays a key role in distinguishing the two cases of Proposition~\ref{prop:PV1}.  In the following subsections we provide the necessary bounds on the integral kernel of $W_s$ to verify both claims in Proposition~\ref{prop:PV1} and formulate a similar proposition about the operator $W_{log}$ defined in \eqref{eqn:Wlog}.  Once these facts are in hand, Theorem~\ref{thm:main} is an immediate consequence.

\subsection{Estimates when $|y| > 2|x|$}
We begin the analysis of $K(x,y)$ in the region where $y$ is large compared to $x$.
As suggested by Lemma~\ref{lem:ybig}, the decay as $|y| \to \infty$
dictates the upper range of exponents for which $W_s$ is bounded on $L^p(\R^4)$.
Our use of the cancellation condition $P_eV1= 0$ (which always holds, see \eqref{eqn:PV1}), and (when it is assumed) $P_eVx=0$
mirrors its treatment in~\cite{GGwaveop}.

We first rewrite the 
$K(x,y)$ integral in the following manner.
\begin{multline}	\label{eqn:KjkPV1}
K(x,y) = \iint_{\R^{8}} \int_0^\infty V\psi(z) V\psi(w)
R_0^+(\lambda^2,|x-z|) \\
\big((R_0^+ - R_0^-)(\lambda^2, |y-w|) - (R_0^+ - R_0^-)(\lambda^2,|y|)\big)
\frac{\tilde\Phi(\lambda)}{\lambda} \,d\lambda \, dz\, dw.
\end{multline}
Subtracting $(R_0^+ - R_0^-)(\lambda^2,|y|)$, which is independent of $w$,
from the integrand does not affect the final value
due to \eqref{eqn:PV1}.

The strong decay of $V\psi(z)$ and $V\psi(w)$, together with compact support of $\tilde{\Phi}(\lambda)$ and 
boundedness as $\lambda \to 0$ (due to~\eqref{eqn:R+minusR-}) allows the order of integration
to be changed freely.

For any function $F(\lambda, |y|)$ one can express
\begin{equation}\label{eqn:Taylor1}
F(\lambda, |y-w|) - F(\lambda,|y|) = \int_0^1 \partial_r F(\lambda, |y-sw|)
\frac{(-w) \cdot (y-sw)}{|y-sw|}\,ds.
\end{equation}
where $\partial_r$ indicates the partial derivative with respect to the radial variable of
$F(\lambda, r)$.
Here we are interested in $F(\lambda, |y|)
= (R_0^+ - R_0^-)(\lambda^2,|y|)$,  
whose radial derivatives are considered in the statement of Lemma~\ref{lem:lambdaInt}.
Identity \eqref{eqn:Taylor1} is most useful in the region of~\eqref{eqn:KjkPV1}
where $|w| < \frac12|y|$.  In this region we use the righthand side of~\eqref{eqn:Taylor1}
to express the contribution to $K(x,y)$ as
\begin{multline} \label{eqn:K s}
\int_{|w| < \frac{|y|}{2}} \int_{\R^4} \int_0^1 \int_0^\infty 
V\psi(z) V\psi(w)
R_0^+(\lambda^2,|x-z|) \\\partial_r\big((R_0^+ -R_0^-)(\lambda^2,|y-sw|)\big)
\frac{(-w)\cdot(y-sw)}{|y-sw|}
\frac{\tilde\Phi(\lambda)}{\lambda}\,ds\,d\lambda \, dz\, dw,
\end{multline}
Applying Fubini's Theorem and then Lemma~\ref{lem:lambdaInt} with $j=1$ we obtain the upper bound
\begin{equation*}
\int_0^1 \int_{|w| < \frac{|y|}{2}} \int_{\R^4} \frac{| V\psi(z)|\, |w V\psi(w)|}{|x-z|^{2}
	\la |x-z| + |y-sw|\ra \la|x-z| - |y-sw|\ra^{ 2}} \,dz\, dw\, ds.
\end{equation*}

By Lemma~\ref{lem:efn decay} and our assumption that $|V(z)| \les \la z\ra^{-4-}$,
we can control the decay of the numerator with $|V\psi(z)| \les \la z\ra^{-7-}$ as our estimate requires that $|w V\psi(w)| \les \la w\ra^{-6-}$.  These are sufficient to apply
Lemma~\ref{lem:A1} in the $z$ variable, then Lemma~\ref{lem:B1} in the $w$ variable to obtain
\begin{multline}\label{eqn:Kbound1}
|\eqref{eqn:K s}| \les \int_0^1 \int_{|w| < \frac{|y|}{2}}  \frac{|w V\psi(w)|}{\la x\ra^{2}\la |x| + |y-sw|\ra
	\la|x| - |y-sw|\ra^{ 2}} \,dw\, ds    \\
\les \int_0^1 \frac{1}{\la x\ra^{ 2}\la|x|+|y|\ra \la |x| - |y|\ra^{ 2}}\,ds 
\les \frac{1}{\la x\ra^{ 2}\la|x|+|y|\ra \la |x| - |y|\ra^{ 2}}. 
\end{multline}
In the region where $|y| > 2|x|$, this yields $|\eqref{eqn:K s}| \les \la x\ra^{-2} \la y\ra^{-3}$.

For the portion of~\eqref{eqn:KjkPV1} where $|w| > \frac12|y|$, 
we do not seek out cancellation between $(R_0^+-R_0^-)(\lambda^2,|y-w|)$ and 
$(R_0^+ - R_0^-)(\lambda^2,|y|)$ and instead treat the two
terms separately.  For the term with $(R_0^+ - R_0^-)(\lambda^2,|y-w|)$,
we have
\begin{multline}
\int_{\R^4}\int_{|w| > \frac{|y|}{2}} \int_0^\infty 
R_0^+(\lambda^2,|x-z|) V\psi(z) V\psi(w)\\ (R_0^+ - R_0^-)(\lambda^2,|y-w|)
\frac{\tilde\Phi(\lambda)}{\lambda} \,d\lambda \, dw\, dz    \\
\les \int_{\R^4} \int_{|w|> \frac{|y|}{2}} \frac{|V\psi(z)|\, |V\psi(w)|}{
	|x-z|^{ 2}\la|x-z| + |y-w|\ra \la|x-z| -|y-w|\ra } \, dw\, dz    \\
\les \frac{1}{\la x\ra^{ 2} \la |x| + |y| \ra \la |x| - |y|\ra \la y\ra}. \label{eqn:Kbound2}
\end{multline}
The first inequality is Lemma~\ref{lem:lambdaInt} with $j=0$.  The second is a combination of
Lemma~\ref{lem:largew} with $k=1$ and $\alpha=0$ for the $w$ integral and Lemma~\ref{lem:A1} for the
$z$ integral.  When $|y| > 2|x|$ this is also bounded by $\la x\ra^{-2}\la y\ra^{-3}$.


The estimate for the term with $(R_0^+ - R_0^-)(\lambda^2,|y|)$ is more straightforward, we obtain the same bound as before,
\begin{multline}
\int_{\R^4}\int_{|w| > \frac{|y|}{2}} \int_0^\infty R_0^+(\lambda^2,|x-z|)V\psi(z) V\psi(w)\\
(R_0^+ - R_0^-)(\lambda^2,|y|)
\frac{\tilde\Phi(\lambda)}{\lambda} \,d\lambda \, dw \, dz  \\
\les \int_{\R^n} \int_{|w|> \frac{|y|}{2}} \frac{|V\psi(z)|\, |V\psi(w)|}{
	|x-z|^{2}\la |x-z| + |y|\ra \la|x-z|-|y|\ra } \, dw\, dz     \\
\les \frac{1}{\la x\ra^{ 2} \la |x|+ |y|\ra  \la |x| - |y|\ra \la y\ra} . \label{eqn:Kbound3}
\end{multline}
This term does not vanish because of the restricted domain of the $w$ integral.
We have used Lemma~\ref{lem:A1} in $z$, and 
the  estimate $\int_{|w| > |y|/2} \la w\ra^{-N}\,dw \les \la y\ra^{4-N}$ (for $N>4$) in 
lieu of Lemma~\ref{lem:largew}.  
Put together, the integral bounds \eqref{eqn:Kbound1}-\eqref{eqn:Kbound3} show that where $|y| > 2|x|$
\begin{equation} \label{eqn:Kboundnocanc}
|K(x,y)| \les \frac{1}{\la x\ra^2\la y\ra^3},
\end{equation}
so its contribution to $W_s$ is bounded on $L^p(\R^4)$ for $1\leq p<4$ by Lemma~\ref{lem:ybig} with $\beta=3$ and
$\gamma = 2$.  

\subsection{Improvement when $P_eVx=0$}
If we assume that $P_eVx=0$, this permits us to introduce a linear approximation of
$(R_0^+ - R_0^-)(\lambda^2, |y-w|)$ in the $w$ variable without changing the value of the
integral~\eqref{eqn:KjkPV1}. That is, we have the equality
\begin{align}
\int_0^\infty &R_0^{+}(\lambda^2,|x-z|) V(z)\psi(z)
V(w) \psi(w) (R_0^+-R_0^-)(\lambda^2,|w-y|)
\frac{\tilde\Phi(\lambda)}{\lambda} \, d\lambda]\nn \\
&=\int_0^\infty R_0^{+}(\lambda^2,|x-z|) V\psi(z)
V\psi(w)\label{eqn:2canc}\\ &\Big[(R_0^+-R_0^-)(\lambda^2,|w-y|) -F(\lambda,y) - G(\lambda,y)\frac{w\cdot y}{|y|}\Big]
\frac{\tilde\Phi(\lambda)}{\lambda} d\lambda\nn
\end{align}
for any functions $F(\lambda, y)$ and $G(\lambda, y)$.  
In place of \eqref{eqn:Taylor1}, we utilize the second
level of cancellation to write
\begin{multline}\label{eqn:K2canc}
K(\lambda, |y-w|) - K(\lambda, |y|) + \partial_rK(\lambda,|y|) \frac{w \cdot y}{|y|}  \\
= \int_0^1 (1-s) \bigg[\partial_r^2 K(\lambda, |y-sw|)  \frac{(w \cdot (y-sw))^2}{|y-sw|^2}\\
+ \partial_r K(\lambda, |y-sw|)\Big(\frac{|w|^2}{|y-sw|} - \frac{(w \cdot (y-sw))^2}{|y-sw|^3}\Big)\bigg]\, ds.
\end{multline}
The formula above suggests that we choose
$F(\lambda, y)=(R_0^+-R_0^-)(\lambda^2,|y|)$ and
$G(\lambda,y)=\partial_r (R_0^+-R_0^-)(\lambda^2,|y|)$
in \eqref{eqn:2canc} respectively.

As in the arguments of the previous section, we use the left side of~\eqref{eqn:K2canc}
when $|w| > |y|/2$ because there is no significant cancellation of these three terms.
One can imitate~\eqref{eqn:Kbound2} more or less exactly
to show that they contribute no more than $\la x\ra^{-2} \la y\ra^{-4}$ to the
size of $K(x,y)$. The assumption $|V\psi(w)|\les \la w\ra^{-7-}$ is sufficient to 
apply Lemma~\ref{lem:largew} with $k=2$ instead of $k=1$ as needed.

The portion of $K(x,y)$ originating from the region $|w|<|y|/2$,
consists of new terms of the  form
\begin{multline*}
\int_{|w|<\frac{|y|}{2}}\int_{\R^4} \int_0^\infty \int_0^1 
V\psi(z) V\psi(w)
R_0^+(\lambda^2,|x-z|)\\
\partial_r^j\big((R_0^+ -R_0^-)(\lambda^2,|y-sw|)\big) 
(1-s)\Gamma_j(s,w,y)
\frac{\tilde\Phi(\lambda)}{\lambda}\,ds\,d\lambda \, dz\,dw
\end{multline*}
with $j=1,2$ and
$\Gamma_j(s,w,y)$ denoting
$$
\Gamma_1 = \Big(\frac{|w|^2}{|y-sw|} - \frac{(w \cdot (y-sw))^2}{|y-sw|^3}\Big), \qquad
\textrm{and} \qquad  \Gamma_2 = \frac{(w \cdot (y-sw))^2}{|y-sw|^2}.
$$  
When $|w| < |y|/2$ and $0 \leq s \leq 1$, these factors obey the bounds 
$|\Gamma_1(s,w,y)| \les |y|^{-1}|w|^2$ and $|\Gamma_2(s,w,y)| \leq |w|^2$.
The calculation proceeds in the same manner as the estimate for~\eqref{eqn:K s}, first
using Lemma~\ref{lem:lambdaInt} with $j=1,2$, then
Lemma~\ref{lem:A1} in the $z$ integral and Lemma~\ref{lem:B1} (with $\alpha = 2-j$) 
in the $w$ integral.  
For both terms we have the bound $\la x\ra^{-2} \la y\ra^{-4}$ when $|y|>2|x|$.

In this case the use of Lemma~\ref{lem:B1} with $\beta = 1+j$
requires that $|w|^2|V\psi(w)| \les \la w\ra^{-7}$, from which 
Lemma~\ref{lem:efn decay} shows that $|V(w)|\les \la w\ra^{-6-}$ is needed.

Put together with the previous claim, this implies that if $P_eVx=0$, then
\begin{equation} \label{eqn:Kboundcanc}
|K(x,y)| \les \frac{1}{\la x\ra^{ 2}\la y\ra^{4}} \ \text{when} \ |y| > 2|x|. 
\end{equation} 
The operator with kernel $K(x,y)$ in this region is therefore 
bounded on $L^p(\R^4)$ for all $p \in [1,\infty)$, by Lemma~\ref{lem:ybig}.

\subsection{Estimates when $|x| \approx |y|$}

No new work is required to control $K(x,y)$ adequately when $x$ and $y$ are
of similar size. We need only use the fact that $P_eV1=0$ as before, then combining~\eqref{eqn:Kbound1} and~\eqref{eqn:Kbound2} when $\frac12|x|\leq |y| \leq 2|x|$
leads to the bound
\begin{equation} \label{eqn:KboundRing}
|K(x,y)| \les \frac{1}{\la x\ra^3\la|x|-|y|\ra^2} + \frac{1}{\la x\ra^4\la|x|-|y|\ra}\les \frac{1}{\la x\ra^3\la|x|-|y|\ra^2}.
\end{equation}
This is bounded on $L^p(\R^4)$ for all $1 \leq p \leq \infty$ by Lemma~\ref{lem:xapproxy}.

\begin{rmk}
	In dimensions $n \geq 5$ it is possible to prove adequate bounds for $K(x,y)$ in
	the region where $|x| \approx |y|$ without assuming that $P_eV1 = 0$.  Here
	the cancellation is essential, as a straightforward estimation of~\eqref{eqn:Ws22}
	gives the bound $|K(x,y)| \les \la x\ra^{-3}\la|x|-|y|\ra^{-1}$ in this region,
	which does not lead to operator bounds on $L^p(\R^4)$ for any exponent $p$. 
\end{rmk}

\subsection{Estimates when $|x| > 2|y|$}
In the region where $|x| > 2|y|$ the combined bounds~\eqref{eqn:Kbound1}
and~\eqref{eqn:Kbound2} imply that
\begin{equation} \label{eqn:KboundXnocanc}
|K(x,y)| \les \frac{1}{\la x\ra^5} + \frac{1}{\la x\ra^4 \la y\ra} \les \frac{1}{\la x\ra^4},
\end{equation}
which according to Lemma~\ref{lem:xbig} gives rise to a bounded operator on $L^p(\R^4)$
for all $p>1$.  It is certainly bounded on $L^1(\R^4)$ as well if we restrict to the compact
region where $2|y| < |x| < 1$.

In order to show that the entire kernel of $W_s$ is bounded on $L^1(\R^4)$,
we must improve the decay as
$|x| \to \infty$ enough to make it integrable.  This is accomplished by applying the
cancellation condition $P_eV1=0$  in the $z$ variable in \eqref{eqn:Ws22}.  Then we may write  
\begin{multline}	\label{eqn:Kleftcanc}
K(x,y) = \iint_{\R^{8}} \int_0^\infty V\psi(z) V\psi(w)
\big(R_0^+(\lambda^2,|x-z|) - R_0^+(\lambda^2,|x|)\big) \\
(R_0^+ - R_0^-)(\lambda^2, |y-w|)\big)
\frac{\tilde\Phi(\lambda)}{\lambda} \,d\lambda \, dz\, dw.
\end{multline}
We split this integral into two regions: where $|z| > \frac12|x|$, and 
where $|z| < \frac12|x|$.

The first region is evaluated in the same spirit as~\eqref{eqn:Kbound2}.
\begin{multline*}
\int_{\R^4}\int_{|z| > \frac{|x|}{2}} \int_0^\infty 
R_0^+(\lambda^2,|x-z|) V\psi(z) V\psi(w)\\
(R_0^+ - R_0^-)(\lambda^2,|y-w|)
\frac{\tilde\Phi(\lambda)}{\lambda} \,d\lambda \, dz\, dw   \notag \\
\les \int_{\R^4} \int_{|z|> \frac{|x|}{2}} \frac{|V\psi(z)|\, |V\psi(w)|}{
	|x-z|^{ 2}\la|x-z| + |y-w|\ra \la|x-z| -|y-w|\ra } \, dz\, dw    \\
\les \frac{1}{\la x\ra^{ 3} \la |x| + |y| \ra \la |x| - |y|\ra} 
\les \frac{1}{\la x\ra^5}.
\end{multline*}
The first inequality is Lemma~\ref{lem:lambdaInt}, and the second
follows from Lemma~\ref{lem:largew} with $\alpha=2$, $k=1$,
and Lemma~\ref{lem:A1}.

The integral with $R_0^+(\lambda^2,|x|)$ is similar.  Assuming $|x|>\max(1, 2|y|)$
we have
\begin{multline*}
\int_{\R^4}\int_{|z| > \frac{|x|}{2}} \int_0^\infty 
R_0^+(\lambda^2,|x|) V\psi(z) V\psi(w) \\ (R_0^+ - R_0^-)(\lambda^2,|y-w|)
\frac{\tilde\Phi(\lambda)}{\lambda} \,d\lambda \, dz\, dw   \notag \\
\les \int_{\R^4} \int_{|z|> \frac{|x|}{2}} \frac{|V\psi(z)|\, |V\psi(w)|}{
	|x|^{ 2}\la|x| + |y-w|\ra \la|x| -|y-w|\ra } \, dz\, dw    \\
\les \frac{1}{\la x\ra^{ 3} \la |x| + |y| \ra \la |x| - |y|\ra} 
\les \frac{1}{\la x\ra^5}.
\end{multline*}
These again follow by Lemmas~\ref{lem:lambdaInt}, \ref{lem:largew}, 
and~\ref{lem:A1}.

Finally the part of~\eqref{eqn:Kleftcanc} that comes from integrating where
$|z| < \frac12|x|$ is evaluated in the same manner as~\eqref{eqn:K s}.
Namely,
\begin{multline*}
\int_{|z| < \frac{|x|}{2}} \int_{\R^4} \int_0^1 \int_0^\infty 
V\psi(z) V\psi(w)
\partial_r\big(R_0^+(\lambda^2,|x-sz|)\big) (R_0^+ -R_0^-)(\lambda^2,|y-w|) \\
\times \frac{(-z)\cdot(x-sz)}{|x-sz|}
\frac{\tilde\Phi(\lambda)}{\lambda}\,ds\,d\lambda \, dw\, dz 
\\
\les
\int_0^1 \int_{|z| < \frac{|x|}{2}} \int_{\R^4} \frac{| z V\psi(z)|\, | V\psi(w)|}{|x-sz|^{3}
	\la|x-sz| - |y-w|\ra^2} \,dw\, dz\, ds
\\
\les \int_0^1 \frac{1}{\la x\ra^3 \la|x|-|y|\ra^2}\,ds \les \frac{1}{\la x\ra^5}.
\end{multline*}
Put together, if $|V(w)|\les \la w\ra^{-6-}$, we conclude that in the entire region where $|x| > 2|y|$,
\begin{equation} \label{eqn:KboundX}
|K(x,y)| \les \frac{1}{\la x\ra^5},
\end{equation}
and this describes a bounded operator on all $L^p(\R^4)$, $1 \leq p \leq \infty$ by Lemma~\ref{lem:xbig}.

\subsection{Estimates for $W_{log}$}
Pointwise estimates for the integral kernel of $W_{log}$ follow from the same
arguments as with $W_s$ above, only with Lemma~\ref{lem:logssuck}
used in place of Lemma~\ref{lem:A1}.

\begin{prop} \label{prop:Wlog}
	If $|V(z)| \les \la z\ra^{-4-}$, then the operator $W_{log}$ is bounded on $L^p(\R^4)$
	for $1\leq p < \infty$.
\end{prop}
\begin{rmk}
	In fact $W_{log}$ is bounded on $L^\infty(\R^4)$ as well.  We omit this case as a matter
	of convenience, in order to present a shorter proof with no reliance on cancellation properties. 
\end{rmk}

\begin{proof}
	The kernel of $W_{log}$ is given by
	\begin{multline*}
	W_{log}(x,y)=\frac{1}{\pi i} \int_{\R^8} \int_0^\infty R_0^+(\lambda^2,|x-z|)
	\bigg[ \sum_{j,k=1}^d a_{jk} V\psi_j(z) V\psi_k(w)
	\bigg]\\ [R_0^+ -R_0^-](\lambda^2,|y-w|)
	\tilde \Phi(\lambda)\lambda \log(\lambda )\, d\lambda\, dz\, dw.
	\end{multline*}
	Applying Lemma~\ref{lem:logssuck} with $j=2$ and $\ell=1$, we have
	\begin{align*}
	|W_{log}(x,y)|&  
	\les  \int_{\R^8}\frac{\la \log \la |x-z|-|y-w|\ra\ra |V\psi_j(z)|\,| V\psi_k(w)|}{|x-z|^{ 2}\la |x-z|+|y-w|\ra \la |x-z|-|y-w|\ra^{3}} \, dz\, dw\\
	&  
	\les  \int_{\R^8}\frac{\la z \ra^{-7-}\la w \ra^{-7-}}{|x-z|^{ 2}\la |x-z|+|y-w|\ra \la |x-z|-|y-w|\ra^{3-\eps}} \, dz\, dw,
	\end{align*}
	for any sufficiently small $\eps>0$.  The last inequality follows from the assumption on the decay of $V$ and Lemma~\ref{lem:efn decay}.  Now, applying Lemma~\ref{lem:A1} in both the $z$ and $w$ integrals, we have the upper bound
	$$
	|W_{log}(x,y)|\les \frac{1}{\la x\ra^2 \la |x|+|y|\ra \la |x|-|y|\ra^{3-\eps}}
	\les \begin{cases} \frac{1}{\la x\ra^{6-\eps}} &\text{ if } |x| > 2|y| \\
	\frac{1}{\la x\ra^3\la|x|-|y|\ra^{3-\eps}} &\text{ if } |x| \approx |y| \\
	\frac{1}{\la x\ra^2\la y\ra^{4-\eps}} &\text{ if } |y| > 2|x|  \end{cases}.
	$$
	Applying Lemmas~\ref{lem:xbig}-\ref{lem:xapproxy} concludes the proof, with Lemma~\ref{lem:ybig}
	specifically showing that $W_{log}$ is bounded on $L^p(\R^4)$ for $1 \leq p < \frac{4}{\eps}$.
\end{proof}

\subsection{The Proof of Theorem~\ref{thm:main}}

We conclude the section by proving the main theorem.

\begin{proof}[Proof of Proposition~\ref{prop:PV1}]
	In the preceding subsections, we have assembled the following
	bounds \eqref{eqn:Kboundnocanc}, \eqref{eqn:KboundRing}, and~\eqref{eqn:KboundX} for the integral kernel of $W_s$:
	$$
	|K(x,y)| \les \begin{cases} \frac{1}{\la x\ra^5} &\text{ if } |x| > 2|y|, \\
	\frac{1}{\la x\ra^3\la|x|-|y|\ra^2} &\text{ if } |x| \approx |y|, \\
	\frac{1}{\la x\ra^2 \la y\ra^3} &\text{ if } |y| > 2|x| \end{cases}\, .
	$$
	The pointwise estimates in the first two regions are sufficient to satisfy the
	Schur test, so they describe bounded operators on $L^p(\R^4)$ for all $1 \leq p \leq \infty$.
	When $|y|$ is large, the integral kernel describes a bounded operator only if
	$1 \leq p < 4$ by Lemma~\ref{lem:ybig}, provided $|V(z)|\les \la z\ra^{-4-}$.
	
	With the further assumption that $P_eVx=0$ and the additional decay of $|V(z)| \les \la x\ra^{-6-}$, 
	we are able to replace~\eqref{eqn:Kboundnocanc} with the stronger pointwise
	bound~\eqref{eqn:Kboundcanc}, which asserts that 
	$|K(x,y)| \les \la x\ra^{-2}\la y\ra^{-4}$ when $|y| > 2|x|$.
	This now describes a bounded operator on $L^p(\R^4)$ in the range
	$1 \leq p < \infty$.
\end{proof}

\begin{proof}[Proof of Theorem~\ref{thm:main}]

	Section~4 of \cite{JY4} provides the decomposition $W = W_< + W_>$, with
	$W_<=\Phi(H)(1-(W_{s}+W_{log}+W_{0}+W_{-1}+W_r))\Phi(H_0)$,
	where
	\begin{align} 
	W_{s}&=\frac{1}{\pi i} \int_0^\infty R_0^+(\lambda^2)
	VP_eV(R_0^+(\lambda^2)-R_0^-(\lambda^2))
	\tilde \Phi(\lambda)\lambda^{-1}\, d\lambda. \tag{\ref{eqn:Ws}} \\
	W_{log}&=\frac{1}{\pi i} \int_0^\infty R_0^+(\lambda^2)
	L_1 (R_0^+(\lambda^2)-R_0^-(\lambda^2))
	\tilde \Phi(\lambda)\lambda \log(\lambda )\, d\lambda.  \tag{\ref{eqn:Wlog}}
	\end{align}
	The operator $W_{log}$ as defined here is called $W_1$ in \cite{JY4}.   The remaining terms $W_0$, $W_{-1}$, and $W_r$ are shown to be bounded on $L^p$ for $1\leq p\leq \infty$ 
	in Lemma~4.1 of \cite{JY4}.
	
	Under the assumptions on the decay of $V$ and
	\eqref{Vfour}, it was shown in \cite{JY4} that
	$W_>$ is also bounded on $L^p(\R^4)$ for $1\leq p\leq \infty$.
	Proposition~\ref{prop:Wlog} asserts that $W_{log}$ is bounded
	on $L^p(\R^4)$ in the range $1 \leq p < \infty$.

	The kernels of $\Phi(H)$ and $\Phi(H_0)$ are bounded
	by $C_N \la x-y \ra^{-N}$ for each $N=1,2,\dots$, see
	Lemma~2.2 of \cite{YajWkp3}.  Following \eqref{eqn:Ws}, 
	\eqref{eqn:Ws22}, $W_{s}$ is
	bounded on $L^p(\R^4)$ exactly when the operators
	$K^{jk}$ are.  The range of exponents $p$ for which this occurs
	is determined by Proposition~\ref{prop:PV1}.
\end{proof}

\begin{rmk}\label{PVx2 rmk}
	
	We note that the endpoint
	$p=\infty$ is not covered in Theorem~\ref{thm:main}.  With the additional
	assumptions $P_eVx^2=0$ and $|V(z)|\les \la z\ra^{-8-}$, 
	one can use the techniques
	above to show that the wave operators are bounded
	on $L^p(\R^4)$ for the full range of $1\leq p\leq \infty$.  Here the assumption $P_eVx^2=0$ means
	that $\int_{\R^n} P_2(x)V(x)\psi(x)\, dx=0$ for
	any quadratic monomial $P_2$. As with the $L^\infty$ bound for $W_{log}$,
	we leave the details to the reader.
	
\end{rmk}

\section{Remarks On Threshold Resonances}\label{sec:res}

Finally, we discuss some aspects of how threshold resonances affect the $L^p$ boundedness of the wave operators.  
There is strong evidence that they are not bounded on $L^p(\R^4)$ for any $p>2$ in this case.
By the intertwining identity~\eqref{eqn:intertwine}, $L^p$ boundedness of wave operators with $p>2$
would imply a power-law decay of the linear evolution $e^{itH}P_{ac}(H)$ at the rate $|t|^{(4/p)-2}$ as in~\eqref{eqn:dispersive}.
Instead it is known that, as an operator from $L^1(\R^4)\to L^\infty(\R^4)$,
$$
e^{itH}P_{ac}(H) = \phi(t)P + O(t^{-1}),
$$
with $\phi(t) \sim (\log t)^{-1}$ and $P$ a finite rank operator, \cite{EGG,GT}. 
The exact form of $P$ depends on whether or not there is a zero energy eigenvalue as well.
We suspect that the wave operators are bounded in $L^p(\R^4)$ precisely in the range $1<p\leq2$. 

The low energy resolvent expansion of $R_V^\pm (\lambda^2)$ is considerably more complicated when a zero energy resonance is present.  See, for example \cite{Jen2,EGG,GT}.  For the sake of (relative) simplicity we consider the case of a resonance and no eigenvalue at zero.  Then the most singular $\lambda$ term is of the form
$$
\frac{P_r(x,y)}{\lambda^2 (a\log \lambda +z)}, \qquad a\in \R \setminus\{0\}, \quad z\in \mathbb C \setminus \R,
$$
where $P_r(x,y)$ is a Riesz projection onto the canonical zero energy resonance.  One crucial feature is that
$P_rV1 \not= 0$ for resonances, so most cancellation arguments do not apply.


A heuristic argument suggests that the leading order term of $W_<$ would take the form
\begin{align*}
\int_{\R^8}
\int_0^\infty R_0^+(\lambda^2,|x-z|) \frac{V\varphi(z) V\varphi(w)}{\lambda (a\log \lambda +z)} [R_0^+-R_0^-](\lambda^2, |y-w|) \, d\lambda\, dz\, dw,
\end{align*}
with $\varphi$ the canonical resonance at zero.
This should have the same order of magnitude as what we calculated for $W_s$, perhaps even better by a logarithmic factor, except that no cancellation may be applied.  If the estimate~\eqref{eqn:KboundXnocanc} can only be improved by a factor of $\log \la x\ra$ for large $|x|$, then the operator may not be bounded in $L^1(\R^4)$.

In the region where $|x| \approx |y|$, the lack of cancellation suggests that the best pointwise bounds for $K(x,y)$ will fall into the poorly behaved $\alpha = \beta = 0$ case of Lemma~\ref{lem:xapproxy}.  Establishing boundedness on $L^p(\R^4)$ for such an operator would require more precise information about the sign and smoothness of $K(x,y)$.  This seems closer in spirit to the approach taken in~\cite{YajNew2}, where wave operator bounds are linked to weighted $L^p$ estimates for the Hilbert transform rather than operators with a positive kernel. 

\appendix
\setcounter{section}{0}
\renewcommand{\thesection}{\Alph{section}}

\section{Integral Estimates}\label{sec:appA}

We begin by proving the assertions in Lemmas~\ref{lem:xbig}-\ref{lem:xapproxy}
regarding the $L^p(\R^4)$ behavior of integral operators with certain pointwise bounds.  We then proceed to prove technical lemmas on integral bounds required to complete our arguments.
\begin{proof}[Proof of Lemma~\ref{lem:xbig}]
	Recall the Schur test, if
	$$
	\sup_x \int_{\R^4}|K(x,y)|\, dy+\sup_y \int_{\R^4}|K(x,y)|\, dx <\infty,
	$$
	then the integral operator with kernel $K(x,y)$ is bounded on $L^p(\R^4)$ for all $1\leq p\leq \infty$.  If $|K(x,y)| \les \la x\ra^{-4-\alpha}$ for $\alpha>0$, the full range of $p$ is attained as
	\begin{align*}
	\sup_x \int_{\R^4} |K(x,y)|\,dy \les \sup_x \la x\ra^{-4-\alpha}\int_{|y|< \frac12|x|} \,dy \les \sup_x \la x\ra^{-\alpha} \les 1.
	\\
	\sup_y \int_{\R^4} |K(x,y)|\,dx \les \sup_y \int_{|x| > 2|y|} \la x\ra^{-4-\alpha}\,dx \les \int_{\R^4} \la x\ra^{-4-\alpha}\,dx
	\les 1.
	\end{align*}
	If $\alpha=0$, then $|K(x,y)| \les \la x\ra^{-4}$. Note that for each $x \in \R^4$,
	$$
	\int_{\R^4}K(x,y)f(y)\,dy \les |x|^{-4}\int_{|y| < |x|} |f(y)|\,dy \les {\mathcal M}f(x),
	$$
	where ${\mathcal M}f$ is the Hardy-Littlewod maximal function.  The claimed range of $p$ follows from the well-known bounds of the Hardy-Littlewood maximal function, see~\cite{Stein} for example.
\end{proof}

\begin{proof}[Proof of Lemma~\ref{lem:ybig}]
	It suffices to prove this for $\gamma = 4-\beta$, as $\la x\ra^{-\gamma} \leq \la x\ra^{\beta-4}$ in all other cases.
	Note that $K$ is bounded on $L^1(\R^4)$ since, if $\beta>0$,
	$$
	\sup_y \int_{\R^4} |K(x,y)|\,dx \les \sup_y \la y\ra^{-\beta}\int_{|x| < \frac12|y|} \la x \ra^{\beta - 4}\,dx
	\les 1.
	$$
	Now consider $0 < \beta < 4$.  For each $x \in \R^4$, the weak-$L^{4/\beta}$ norm of $K(x,\,\cdot\,)$
	is uniformly bounded by $\la x\ra^{\beta-4}$.  Hence $K$ defines a bounded operator between the Lorentz spaces
	$L^{\frac{4}{4-\beta},1}(\R^4)$ and $L^{\frac{4}{4-\beta},\infty}(\R^4)$.  Boundedness on $L^p(\R^4)$, $1 \leq p <\frac{4}{4-\beta}$ follows by interpolation, \cite{hunt}.
	The case $\beta = 4$ is the dual formulation of Lemma~\ref{lem:xbig}, hence $K$ defines a bounded operator on $L^p(\R^4)$ for all $1 \leq p < \infty$.
\end{proof}

\begin{proof}[Proof of Lemma~\ref{lem:xapproxy}]
	Once again we can use the Schur test.  This time we integrate  in spherical coordinates for $y$:
	\begin{multline*} 
	\sup_x \int_{\R^4}|K(x,y)|\,dy \les \sup_x \la x\ra^{-3-\alpha}\int_{|y| \approx |x|}\la|x|-|y|\ra^{-1-\beta}\,dy
	\\
	\les \sup_x \la x\ra^{-\alpha} \int_{\frac12|x|}^{2|x|} \la |x|-r\ra^{-1-\beta} \,dr
	\les \begin{cases} \sup_x \la x\ra^{-\alpha} & \text{ if } \beta> 0 \\
	\sup_x \la x \ra^{-\alpha} (1+ \log |x|) & \text{ if } \beta = 0 \end{cases}.
	\end{multline*}
	These are bounded by 1 unless $\alpha$ and $\beta$ are both zero.  By symmetry the result is the same
	if $x$ and $y$ are reversed.
	
	If $\alpha = \beta = 0$, let $f$ be a characteristic function of the annulus $R < |y| < 2R$ with radius $R > 1$.  Then for every
	point $x$ in the same annulus, $\int_{\R^4} |K(x,y)| f(y)\, dy \gtrsim \log \la R\ra$.
\end{proof}

The following integral estimates, along with Lemma~\ref{lem:lambdaInt} and Lemma~\ref{lem:A1}, are direct quotes or minor adaptations of statements in~\cite{GGwaveop}.
Proofs are presented where they differ from the previous work.

\begin{lemma}\label{lem:B1}
	
	Suppose  $0<s\leq 1$, $0\leq\alpha\leq n$, $\beta \geq 1$, and $\gamma \in \R$.
	If $N\geq n+\beta$, then for each fixed constant $R\geq 0$, we have the bound
	\begin{multline}  \label{eqn:B1}
	\int_{|w| < \frac{|y|}{2}}\frac{\la w\ra^{-N}}{|y-sw|^\alpha\la |y-sw|+R\ra^\gamma \la |y-sw|-R\ra^\beta} \, dw\\
	\les \frac{1}{\la y\ra^\alpha \la |y|+R\ra^\gamma \la R-|y| \ra^\beta}.
	\end{multline}
	
\end{lemma}

\begin{proof}
	
	The claim for large $|y|>10$ essentially follows from Lemma~4.4 in \cite{GGwaveop}, which considers only the case $\gamma = 1$.  However one of the first steps is to observe that $\la|y-sw| + R\ra \approx \la|y|+R\ra$ for all $|w| < \frac12|y|$ and $|s| \leq 1$, so the calculations proceed independently of the choice of $\gamma$.  To complete the proof when $|y|\leq 10$, we simply note that the region of integration has volume comparable to $|y|^n$ and the integrand is nearly constant.  Thus
	\begin{multline*}
	\int_{|w| < |y|/2}\frac{\la w\ra^{-N}}{|y-sw|^\alpha\la |y-sw|+R\ra^\gamma \la |y-sw|-R\ra^\beta} \, dw\\ \les \frac{|y|^{n-\alpha}}{ \la |y|+R\ra^\gamma \la R-|y| \ra^\beta}.
	\end{multline*}
\end{proof}

\begin{lemma} \label{lem:largew}
	Let $k \geq 0$ and $0 \leq \alpha < n-1$.
	If $N \geq n+1+k$ and $R \geq 0$ is fixed, then
	\begin{equation} \label{eqn:largew}
	\int_{|w| > \frac{|y|}{2}}\frac{\la w\ra^{-N}}{\la R+|y-w|\ra \la R-|y-w|\ra |y-w|^{\alpha}} \, dw
	\les \frac{1}{\la R+|y| \ra \la R-|y| \ra \la y\ra^{\alpha +k}}.
	\end{equation}
\end{lemma}

\begin{proof}
	The $\alpha=0$ case is proved as Lemma~3.2 in~\cite{GGwaveop}.
	If $\alpha >0$, split the domain of integration into two pieces according to whether 
	$|y-w| > \frac12|y|$ or $|y-w| < \frac12|y|$.  In the former region, 
	$|w|$, and $|y-w|$ are of comparable size, so one may reduce to the $\alpha = 0$ case
	$$
	\int_{|w| > \frac{|y|}{2}}\frac{\la w\ra^{-N}|w|^{-\alpha}}{\la R+|w|\ra \la R-|w|\ra} \, dw
	\les |y|^{-\alpha}	\int_{|w| > \frac{|y|}{2}}\frac{\la w\ra^{-N}}{\la R+|w|\ra \la R-|w|\ra} \, dw,
	$$
	which suffices if $|y| > 1$.  If $|y| < 1$, we bound the integral by
	\begin{multline*}
	\int_{\R^4} \frac{\la w\ra^{-N}|w|^{-\alpha}}{\la R+|w|\ra \la R-|w|\ra} \, dw \\
	\les \int_{|w|<1}\frac{|w|^{-\alpha}}{\la R\ra^2} + \int_{|w| > 1}\frac{\la w\ra^{-N}}{\la R+|w|\ra \la R-|w|\ra}
	\les \frac{1}{\la R\ra^2}.
	\end{multline*}
	
	The region where $|y-w| < \frac12|y|$ is best integrated in spherical coordinates 
	centered at the point $y$.  This leads to the expression
	\begin{equation} \label{eqn:shells}
	\int_0^{\frac12|y|} \frac{1}{r^\alpha \la r+R\ra \la r-R \ra}\int_{|y-w|=r}\la w\ra^{-N}\,dw\,dr,
	\end{equation}
	which is precisely handled in~\cite[equation (31)]{GGwaveop} provided $\alpha < n-1$, with the bound 
	$$
	|\eqref{eqn:shells}| \les \frac{1}{\la y\ra^{N-n-1+\alpha} \la|y|+R\ra \la R-|y|\ra}.
	$$
\end{proof}

\begin{lemma}[\cite{GGwaveop}, Lemma 5.2] \label{lem:logssuck}
	
	Let $R_0^\pm(\lambda^2,A)$ denote the convolution kernel of $R_0^\pm(\lambda^2)$
	evaluated at a point with $|x| = A$.  For each $j \geq 0$,
	\begin{multline}
	\int_0^\infty R_0^+(\lambda^2, A)(R_0^+ - R_0^-\big)(\lambda^2,B)
	\lambda^{j-1}(\log \lambda)^\ell \tilde\Phi(\lambda)\, d\lambda\\
	\les \frac{\la \log \la A-B\ra\ra^\ell}{A^{n-2}\la A+B\ra \la A-B\ra^{n-3+j}}.
	\end{multline}
	
\end{lemma}

Lemmas~\ref{lem:lambdaInt} and~\ref{lem:logssuck} rely on the following oscillatory integral estimate.
\begin{lemma}[\cite{GGwaveop}, Lemma 4.1] \label{lem:IBP}
	Suppose there exists $\beta > -1$ and $M > \beta+1$
	such that $|F^{(k)}(\lambda)| \les \lambda^{\beta - k}$ for all $0 \leq k \leq M$.
	Then given a smooth cutoff function $\tilde\Phi$,
	\begin{equation} \label{eqn:IBP}
	\Big|\int_0^\infty e^{i\rho \lambda}F(\lambda) \tilde\Phi(\lambda)\,d\lambda \Big| 
	\les \la \rho\ra^{-\beta-1}.
	\end{equation}
	If $F$ is further assumed to be smooth and supported in the annulus $L \les \lambda \les 1$
	for some $L > \rho^{-1}>0$, then
	\begin{equation} \label{eqn:IBP2}
	\Big|\int_0^\infty e^{i\rho \lambda}F(\lambda) \tilde\Phi(\lambda)\,d\lambda \Big| \les \la \rho\ra^{-M}L^{\beta+1-M}.
	\end{equation}
\end{lemma}

One additional variation is needed specifically to assist in the proof of~\eqref{eqn:KboundX}
for large $x$.

\begin{lemma}\label{lem:leftcanc1}
	
	In four dimensions we have the bound
	\begin{equation}\label{eqn:leftcanc1}
	\int_0^\infty \partial_A R_0^+(\lambda^2, A)\big(R_0^+ - R_0^-\big)(\lambda^2,B)
	\lambda^{-1}\tilde\Phi(\lambda)\, d\lambda
	\les \begin{cases} \frac{1}{A^{ 3} \la A\ra^{2}} & \text{ if } A > 2B \\
	\frac{1}{A^{3} \la B\ra^{2}} & \text{ if } B > 2A \\
	\frac{1}{A^{3}  \la A - B\ra^2 }
	& \text{ if } A \approx B
	\end{cases} .
	\end{equation}
	
\end{lemma}

\begin{proof}

	We recall the expansion for the free resolvent, see Section~4 of \cite{GGwaveop},
	\begin{align}
	R_0^\pm(\lambda^2, A) &= \frac{1}{A^{ 2}}\Omega (\lambda A) + \frac{e^{\pm i\lambda A}}{A^{ 2}}
	\Psi_{\frac{1}{2}}(\lambda A), 
	\label{eqn:asymptotics1}
	\end{align} where $\Omega$ is a bounded compactly supported function that is smooth everywhere except possibly at zero,
	and each $\Psi_{\frac{1}{2}}$ is a smooth function
	supported outside the unit interval  that asymptotically behaves like $(\,\cdot\, )^{1/2}$
	and whose $k^{th}$ derivative behaves like $(\,\cdot\,)^{(1-2k)/2}$.
	This expansion follows from writing $R_0^\pm(\lambda^2,r)= \frac{-\lambda}{8\pi r} [ Y_1(\lambda r)\mp iJ_1(\lambda r)  ]$
	and using the asymptotic properties of the Bessel functions,  see \cite{AS}.
	Additional properties include
	\begin{align}
	\label{eqn:R+minusR-}
	R_0^+(\lambda^2, B) - R_0^-(\lambda^2, B) &= \lambda^{ 2}\Omega(\lambda B)
	+ \frac{e^{i\lambda B}}{B^{ 2}} \Psi_{\frac{1}{2}}(\lambda B)
	+ \frac{e^{-i\lambda B}}{B^{ 2}} \Psi_{\frac{1}{2}}(\lambda B),\\
	\partial_A 	R_0^\pm(\lambda^2, A) &=\frac{e^{i\lambda A}}{A^{ 2}} \lambda  \Psi_{\frac{1}{2}}(\lambda A)+\frac{1}{A^3} \widetilde \Omega(\lambda A).
	\end{align}
	Here $\widetilde \Omega (\lambda A)$ is a compactly supported function that satisfies $|\partial_\lambda^j \widetilde \Omega(\lambda A)|\les \lambda^{-j} $. 
	
	Using the bounds proven in Lemma~2.2 in \cite{GGwaveop} which were considered when $j=1$ (and $n=4$), we have
	\begin{multline*}
	\int_0^\infty \frac{e^{i\lambda A}}{A^{ 2}}   \Psi_{\frac{1}{2}}(\lambda A)
	\left(\lambda^{ 2}\Omega(\lambda B)
	+ \frac{e^{\pm i\lambda B}}{B^{ 2}} \Psi_{\frac{1}{2}}(\lambda B)\right)
	\tilde\Phi(\lambda)\, d\lambda\\ \les
	\frac{1}{A^2 \la A+B\ra \la A-B\ra^2}.
	\end{multline*}
	This satisfies the desired bounds, thus
	we need only bound the contribution of the terms
	$$
	\int_0^\infty \frac{1}{A^3}\widetilde \Omega(\lambda A) \left( \lambda^2  \Omega (\lambda B) + \frac{e^{\pm i\lambda B}}{B^{ 2}} \Psi_{\frac{1}{2}}(\lambda B)
	\right) \, d\lambda.
	$$
	The first term is supported on $\lambda \les \min(1, A^{-1}, B^{-1})$, and hence is bounded by $A^{-3}\min (\la A\ra^{-2} , \la B\ra^{-2} )$.
	Using Lemma~\ref{lem:IBP} with $\rho=B$ and $L=\la B\ra^{-1}$, the second term is bounded by $A^{-3}\la B\ra^{-2}$ and is zero unless $A\les B$.  
\end{proof}

\end{document}